\documentclass[reqno,12pt]{amsart}
\usepackage{amsfonts}
\usepackage{bbm}
\usepackage{}
\numberwithin{equation}{section}
\usepackage{color}
\usepackage{graphicx}
\usepackage{indentfirst}
\usepackage{color}
\usepackage{amssymb}
\usepackage{mathrsfs}
\usepackage{xy}
\usepackage{tikz}

\xyoption{all}
 \def\Hom{\mbox{\rm Hom}} \def\dim{\mbox{\rm dim}\,} \def\Iso{\mbox{\rm Iso}\,}
    \def\mod{\mbox{\rm \textbf{mod}}\,}
    
\def\cone{\mbox{\rm cone}}
\def\ind{\mbox{\rm ind}\,}\def\cim{\mbox{\rm sim}}\def\add{\mbox{\rm add}}
 
\def\Filt{\mbox{\rm \textbf{Filt}}}
\def\f-tors{\mbox{\rm f-tors}} \def\s\tau-tilt{\mbox{\rm s\tau-tilt}}

\theoremstyle{plain}
\newtheorem{theorem}{\bf Theorem}[section]
\newtheorem{lemma}[theorem]{\bf Lemma}
\newtheorem{corollary}[theorem]{\bf Corollary}
\newtheorem{proposition}[theorem]{\bf Proposition}

\newtheorem*{theorem A}{\bf Theorem A}
\newtheorem*{theorem B}{\bf Theorem B}
\newtheorem*{theorem C}{\bf Theorem C}
\newtheorem*{theorem D}{\bf Theorem D}
\newtheorem*{theorem E}{\bf Theorem E}
\theoremstyle{definition}
\newtheorem{definition}[theorem]{\bf Definition}
\newtheorem{remark}[theorem]{\bf Remark}
\newtheorem{example}[theorem]{\bf Example}

\newcommand{\bt}{\begin{theorem}}
	\newcommand{\et}{\end{theorem}}
\newcommand{\bl}{\begin{lemma}}
	\newcommand{\el}{\end{lemma}}
\newcommand{\bd}{\begin{definition}}
	\newcommand{\ed}{\end{definition}}
\newcommand{\bc}{\begin{corollary}}
	\newcommand{\ec}{\end{corollary}}
\newcommand{\bp}{\begin{proof}}
	\newcommand{\ep}{\end{proof}}
\newcommand{\bx}{\begin{example}}
	\newcommand{\ex}{\end{example}}
\newcommand{\br}{\begin{remark}}
	\newcommand{\er}{\end{remark}}
\newcommand{\be}{\begin{equation}}
	\newcommand{\ee}{\end{equation}}
\newcommand{\ba}{\begin{align}}
	\newcommand{\ea}{\end{align}}
\newcommand{\bn}{\begin{enumerate}}
	\newcommand{\en}{\end{enumerate}}
\newcommand{\bcs}{\begin{cases}}
	\newcommand{\ecs}{\end{cases}}

\newcommand{\RNum}[1]{\uppercase\expandafter{\romannumeral #1\relax}}

\makeatletter
\renewcommand{\section}{\@startsection{section}{1}{0mm}
	{-\baselineskip}{0.5\baselineskip}{\bf\leftline}}
\makeatother

\begin{document}
	\title[The first Brauer-Thrall conjecture for extriangulated length categories]{The first Brauer-Thrall conjecture for extriangulated length categories}
	\author[L. Wang, J. Wei]{Li Wang, Jiaqun Wei}
	\address{School of Mathematics-Physics and Finance, Anhui Polytechnic University, 241000 Wuhu, Anhui, P. R. China \endgraf}
	\email{wl04221995@163.com {\rm (L. Wang)}}
	\address{School of Mathematical Sciences, Zhejiang Normal University, 321004 Jinhua, Zhejiang, P. R. China\endgraf}
	\email{weijiaqun5479@zjnu.edu.cn {\rm (J. Wei)}}

	
	
	\begin{abstract} Let $(\mathcal{A},\Theta)$ be a length category. We introduce the notation of Gabriel-Roiter measure with respect to $\Theta$ and extend  Gabriel's main property to this setting.  Using this measure, when $(\mathcal{A},\Theta)$ satisfies some technical conditions, we prove that $\mathcal{A}$ has an infinite number of pairwise nonisomorphic indecomposable objects if and only if it has indecomposable objects of arbitrarily large length. That is,  the first Brauer-Thrall conjecture holds.
	\end{abstract}
	

	\maketitle
	\section{Introduction}
	The first Brauer-Thrall conjecture states that a finite dimension algebra $\Lambda$  is of infinite representation type (i.e. $\mod \Lambda$ has an infinite number of pairwise nonisomorphic indecomposable representations) if and only if $\Lambda$ is of unbounded representation type (i.e. $\mod \Lambda$ has indecomposable representations of arbitrarily large composition length). This conjecture was proved by Roiter in \cite{Ro} by constructing a function that assigns natural numbers to indecomposable modules of finite length.
	
	An abelian category is a  length category if every objects has a finite composition series. The notation of the {\em Roiter measure} was introduced by Gabriel in \cite{Ga} for abelian length categories, which is  a formalization of the induction scheme used in Roiter's proof. 
	
	Both Roiter and Gabriel have assumed from the beginning that there is an upper bound for lengths of indecomposable objects.  Ringle noticed that the formalism of Roiter and Gabriel works as well for arbitrary artin algebra having unbounded representation type. To clarify this matter, Ringle \cite{Ri} defined and studied the  {\em Gabriel-Roiter measure} for any artin algebra. He showed that an  artin algebra of infinite representation type has an infinite chain of Gabriel-Roiter measures. In this manner, two different proofs of the first Brauer-Thrall conjecture were provided.
	
	Although the Gabriel-Roiter measure arises in representation theory, it is actually purely  combinatorial. Krause in \cite{Kr} established an axiomatic characterization of the Gabriel-Roiter measure for a given abelian length category.  In this setting, the Gabriel-Roiter measure $l^{\ast}$ with respect to a length function $l$  is a chain length function of partially ordered sets. This function $l^{\ast}$ process the isomorphism classes of indecomposable objects into  finite chains of $\mathbb{R}_{\geq0}$, which is an appropriate refinement of the composition length. Later, Br$\rm \ddot{u}$stle, Hassoun, Langford and Roy in \cite{Br} investigated the Gabriel-Roiter measure for a finite exact category and how it changes  under reduction of exact structures.

	Besides the  works mentioned, the Gabriel-Roiter measure also has been used in many mathematics, e.g.  Auslander-Reiten theory \cite{Ch2}, wild representation type and GR segment \cite{Ch}, Ziegler spectrum \cite{He}, thin representations \cite{Kra} and so on.
	
	Recently, Nakaoka and Palu \cite{Na} introduced the notion of extriangulated categories by extracting properties on triangulated categories and exact categories.  In \cite{Wa}, the authors defined what they called {\em extriangulated length categories}  as a generalization of abelian length categories. It was proved in \cite[Theorem 3.14]{Wa} that these categories correspond precisely to those extriangulated categories generated by simple-minded systems (see Definition \ref{D-2-5}).  They also  provided a unified framework for studying the torsion classes and $\tau$-tilting theory in this setting.
	
	For an extriangulated length category $\mathcal{A}$, there exists a length function $\Theta$ from the set  of isomorphism classes of objects in $\mathcal{A}$ to $\mathbb{N}$ (see Definition \ref{D-1}). The value $\Theta(M)$ is called a  length of $M$ for any $M\in\mathcal{A}$. For instance, if $\mathcal{A}$ is an abelian length category, then $\Theta(M)$  can be defined by the length of the composition series of $M$ (see Example \ref{E-2-9}). This naturally presents the following question:
	
	\vspace{1.5mm}
	$\mathbf{Question.}$ If the length of indecomposable objects in $\mathcal{A}$ has an upper bound, can we deduce that  $\mathcal{A}$ has a finite number of pairwise nonisomorphic indecomposable objects? That is, the first Brauer-Thrall conjecture is valid?
	\vspace{1.5mm}
	
	To resolve this, we investigate the Gabriel-Roiter measure in the setting of extriangulated length categories. Our strategy is to construct a chain of  measures know as the {\em Gabriel-Roiter chain}, which gives a partition of the isomorphism classes of indecomposable objects. By using this, we prove that the first Brauer-Thrall conjecture  holds for extriangulated length categories of finite type (see Theorem \ref{main2}). For the  case of infinite type, we provide a counter-example (see Example \ref{ex}).
	\vspace{2mm}
	
	$\mathbf{Organization.}$ This paper is organized as follows.  In Section 2, we summarize the definitions and characteristics of the extriangulated (length) categories, providing the foundation for subsequent discussions. Section 3 introduces the notion of  Gabriel-Roiter measure for a given extriangulated length category. We obtain an axiomatic characterization of Gabriel-Roiter measure and use it to  prove the Gabriel's main property. In section 4,  we provide a comprehensive answer for the first Brauer-Thrall conjecture.
	\vspace{2mm}

	$\mathbf{Conventions~and~Notation.}$ Throughout this paper, we assume that all considered categories are skeletally small and Krull-Schmidt, and that the subcategories are full and closed under isomorphisms.

	\section{Preliminaries}
	In this section, we collect some basic definitions and properties of extriangulated (length) categories from \cite{Na} and \cite{Wa}, which
	we use throughout this paper.
	
	\subsection{Extriangulated categories} Let  $\mathcal{A}$ be an additive category with a biadditive functor $\mathbb{E}:\mathcal{A}^{\rm op}\times \mathcal{A}\rightarrow Ab$, where  $Ab$ is the category of abelian groups. For any $A,C\in\mathcal{A}$, an element $\delta\in\mathbb{E}(C,A)$ is called an $\mathbb{E}$-{\em extension}. The zero element $0\in\mathbb{E}(C,A)$ is called the {\em split $\mathbb{E}$-extension}. For any morphism $a\in \Hom_{\mathcal{A}}(A,A')$ and $c\in \Hom_{\mathcal{A}}(C',C)$, we have $\mathbb{E}$-extensions
	$$\mathbb{E}(C,a)(\delta)\in\mathbb{E}(C,A')~\text{and}~\mathbb{E}(c,A)(\delta)\in\mathbb{E}(C',A).$$
	We simply denote them by $a_{\ast}\delta$ and $c^{\ast}\delta$,~respectively. Let $\delta\in\mathbb{E}(C,A)$, $\delta'\in\mathbb{E}(C',A')$ be any pair of $\mathbb{E}$-extensions. A morphism $(a,c)$: $\delta\rightarrow\delta'$ of $\mathbb{E}$-extensions is a pair of morphisms $a\in \Hom_{\mathcal{A}}(A,A')$ and $c\in \Hom_{\mathcal{A}}(C,C')$ satisfying the equality $a_{\ast}\delta=c^{\ast}\delta'$. By the biadditivity of $\mathbb{E}$, we have a natural isomorphism
	$$\mathbb{E}(C\oplus C',A\oplus A')\cong\mathbb{E}(C,A)\oplus\mathbb{E}(C,A')\oplus\mathbb{E}(C',A)\oplus\mathbb{E}(C',A').$$
	Let $\delta\oplus \delta'\in\mathbb{E}(C\oplus C',A\oplus A')$ be the element corresponding to $(\delta,0,0,\delta')$ through
	this isomorphism. Two sequences of morphisms  $A\stackrel{x}{\longrightarrow}B\stackrel{y}{\longrightarrow}C$ and $A\stackrel{x'}{\longrightarrow}B'\stackrel{y'}{\longrightarrow}C$ in $\mathcal{A}$ are said to be {\em equivalent} if there exists an isomorphism $b\in \Hom_{\mathcal{A}}(B,B')$ such that the following diagram is commutative.
	$$\xymatrix{
		A \ar@{=}[d] \ar[r]^-{x} & B\ar@{-->}[d]_{b}^-{\simeq} \ar[r]^-{y} & C\ar@{=}[d] \\
		A \ar[r]^-{x'} &B' \ar[r]^-{y'} &C  }$$ 
	We denote the equivalence class of $A\stackrel{x}{\longrightarrow}B\stackrel{y}{\longrightarrow}C$ by $[A\stackrel{x}{\longrightarrow}B\stackrel{y}{\longrightarrow}C]$. For any $A,C\in\mathcal{A}$,  we denote as
	$$0=[A\stackrel{1\choose0}{\longrightarrow}A\oplus C\stackrel{(0~1)}{\longrightarrow}C].$$
	For any two classes $[A\stackrel{x}{\longrightarrow}B\stackrel{y}{\longrightarrow}C]$ and $[A'\stackrel{x'}{\longrightarrow}B'\stackrel{y'}{\longrightarrow}C']$, we denote as
	$$[A\stackrel{x}{\longrightarrow}B\stackrel{y}{\longrightarrow}C]\oplus[A'\stackrel{x'}{\longrightarrow}B'\stackrel{y'}{\longrightarrow}C']=
	[A\oplus A'\stackrel{x\oplus x'}{\longrightarrow}B\oplus B'\stackrel{y\oplus y'}{\longrightarrow}C\oplus C'].$$
	
	\begin{definition} 	 Let $\mathfrak{s}$ be a correspondence which associates an equivalence class $\mathfrak{s}(\delta)=[A\stackrel{x}{\longrightarrow}B\stackrel{y}{\longrightarrow}C]$ to any $\mathbb{E}$-extension $\delta\in\mathbb{E}(C,A)$. We say $\mathfrak{s}$ is a {\em realization} of $\mathbb{E}$  if it satisfies the following condition ($\ast$). In this case, we say that  sequence  $A\stackrel{x}{\longrightarrow}B\stackrel{y}{\longrightarrow}C$ realizes $\delta$, whenever it satisfies $\mathfrak{s}(\delta)=[A\stackrel{x}{\longrightarrow}B\stackrel{y}{\longrightarrow}C]$.
		
		($\ast$) Let $\delta\in\mathbb{E}(C,A)$ and $\delta'\in\mathbb{E}(C',A')$ be any pair of $\mathbb{E}$-extensions, with $\mathfrak{s}(\delta)=[A\stackrel{x}{\longrightarrow}B\stackrel{y}{\longrightarrow}C]$, $\mathfrak{s}(\delta')=[A'\stackrel{x'}{\longrightarrow}B'\stackrel{y'}{\longrightarrow}C']$. Then for any morphism $(a,c)$: $\delta\rightarrow\delta'$, there exists a morphism $b\in \Hom_{\mathcal{A}}(B,B')$ such that the following diagram commutative.
		$$\xymatrix{
			A \ar[d]_{a} \ar[r]^-{x} & B\ar@{-->}[d]_{b} \ar[r]^-{y} & C\ar[d]_c \\
			A' \ar[r]^-{x'} &B' \ar[r]^-{y'} &C'  }$$ 
		In the above situation, we say that the triplet $(a,b,c)$ realizes $(a,c)$. A realization $\mathfrak{s}$ of $\mathbb{E}$ is said to be {\em additive} if it satisfies the following conditions:
		
		(a) For any $A,C\in\mathcal{A}$, the split $\mathbb{E}$-extension $0\in\mathbb{E}(C,A)$ satisfies $\mathfrak{s}(0)=0$.
		
		(b) $\mathfrak{s}(\delta\oplus\delta')=\mathfrak{s}(\delta)\oplus\mathfrak{s}(\delta')$ for any pair of $\mathbb{E}$-extensions $\delta$ and $\delta'$.
	\end{definition}
	
	\begin{definition}
		(\cite[Definition 2.12]{Na})\label{F}
		We call the triplet $(\mathcal{A}, \mathbb{E},\mathfrak{s})$ an {\em extriangulated category} if it satisfies the following conditions:\\
		$\rm(ET1)$ $\mathbb{E}$: $\mathcal{A}^{\rm op}\times\mathcal{A}\rightarrow Ab$ is a biadditive functor.\\
		$\rm(ET2)$ $\mathfrak{s}$ is an additive realization of $\mathbb{E}$.\\
		$\rm(ET3)$ Let $\delta\in\mathbb{E}(C,A)$ and $\delta'\in\mathbb{E}(C',A')$ be any pair of $\mathbb{E}$-extensions, realized as
		$\mathfrak{s}(\delta)=[A\stackrel{x}{\longrightarrow}B\stackrel{y}{\longrightarrow}C]$, $\mathfrak{s}(\delta')=[A'\stackrel{x'}{\longrightarrow}B'\stackrel{y'}{\longrightarrow}C']$. For any commutative square 
		$$\xymatrix{
			A \ar[d]_{a} \ar[r]^{x} & B \ar[d]_{b} \ar[r]^{y} & C \\
			A'\ar[r]^{x'} &B'\ar[r]^{y'} & C'}$$
		in $\mathcal{A}$, there exists a morphism $(a,c)$: $\delta\rightarrow\delta'$ which is realized by $(a,b,c)$.\\
		$\rm(ET3)^{\rm op}$~Dual of $\rm(ET3)$.\\
		$\rm(ET4)$~Let $\delta\in\mathbb{E}(D,A)$ and $\delta'\in\mathbb{E}(F,B)$ be $\mathbb{E}$-extensions realized by
		$A\stackrel{f}{\longrightarrow}B\stackrel{f'}{\longrightarrow}D$ and $B\stackrel{g}{\longrightarrow}C\stackrel{g'}{\longrightarrow}F$, respectively.
		Then there exist an object $E\in\mathcal{A}$, a commutative diagram
		\begin{equation*}
			\xymatrix{
				A \ar@{=}[d]\ar[r]^-{f} &B\ar[d]_-{g} \ar[r]^-{f'} & D\ar[d]^-{d} \\
				A \ar[r]^-{h} & C\ar[d]_-{g'} \ar[r]^-{h'} & E\ar[d]^-{e} \\
				& F\ar@{=}[r] & F   }
		\end{equation*}
		in $\mathcal{A}$, and an $\mathbb{E}$-extension $\delta''\in \mathbb{E}(E,A)$ realized by $A\stackrel{h}{\longrightarrow}C\stackrel{h'}{\longrightarrow}E$, which satisfy the following compatibilities:\\
		$(\textrm{i})$ $D\stackrel{d}{\longrightarrow}E\stackrel{e}{\longrightarrow}F$ realizes $\mathbb{E}(F,f')(\delta')$,\\
		$(\textrm{ii})$ $\mathbb{E}(d,A)(\delta'')=\delta$,\\
		$(\textrm{iii})$ $\mathbb{E}(E,f)(\delta'')=\mathbb{E}(e,B)(\delta')$.\\
		$\rm(ET4)^{\rm op}$ Dual of $\rm(ET4)$.
	\end{definition}
	
	In what follows, we always assume that $\mathcal{A}:=(\mathcal{A}, \mathbb{E},\mathfrak{s})$ is an extriangulated category. We will use the following terminology.
	
	$\bullet$ Given $\delta\in\mathbb{E}(C,A)$, if $\mathfrak{s}(\delta)=[A\stackrel{x}{\longrightarrow}B\stackrel{y}{\longrightarrow}C]$, then the sequence $A\stackrel{x}{\longrightarrow}B\stackrel{y}{\longrightarrow}C$ is called a {\em conflation}, $x$ is called an {\em inflation} and $y$ is called a {\em deflation}. In this case, we call $$A\stackrel{x}{\longrightarrow}B\stackrel{y}{\longrightarrow}C\stackrel{\delta}\dashrightarrow$$
	is an $\mathbb{E}$-triangle and denote $C=\cone(x)$.
	
	$\bullet$  Let $\mathcal{T},\mathcal{F}$ be two subcategories of $\mathcal{A}$. We define
	$$\mathcal{T}^{\perp}=\{M\in\mathcal{A}~|~\Hom_{\mathcal{A}}(X,M)=0~\text{for any}~X\in\mathcal{T}\},$$
	$$^{\perp}\mathcal{T}=\{M\in\mathcal{A}~|~\Hom_{\mathcal{A}}(M,X)=0~\text{for any}~X\in\mathcal{T}\},$$
	$$\mathcal{T}\ast \mathcal{F}=\{M\in \mathcal{A}~|~\text{there exists an $\mathbb{E}$-triangle}~T\stackrel{}{\longrightarrow}M\stackrel{}{\longrightarrow}F\stackrel{}\dashrightarrow~\text{with}~T\in\mathcal{T},F\in\mathcal{F} \}.$$
	We say a subcategory $\mathcal{C}$ of $\mathcal{A}$ is {\em extension-closed} if $\mathcal{C}\ast \mathcal{C}\subseteq \mathcal{C}$. 
	
	\begin{example}\label{E-2-3} Both exact categories and triangulated categories are typical examples of extriangulated categories. Besides, we may regard an extension-closed subcategory  of $\mathcal{A}$ as an extriangulated category. Explicitly, we have the following observations.
		
		(1) Let $\mathcal{A}$ be an exact category. For any $A,C\in\mathcal{A}$, we define $\mathbb{E}(C,A)$ to be the collection of all equivalence classes of short exact sequences of the form $A\stackrel{}{\longrightarrow}B\stackrel{}{\longrightarrow}C.$ For any $\delta\in\mathbb{E}(C,A)$, define the realization $\mathfrak{s}(\delta)$ to be $\delta$ itself. Then $(\mathcal{A}, \mathbb{E},\mathfrak{s})$ is an extriangulated category. We refer to \cite[Example 2.13]{Na} for more details.

		(2) Let $\mathcal{T}$ be a triangulated category with shift functor $[1]$. For any $A,C\in\mathcal{T}$, we define $\mathbb{E}(C,A):=\Hom_{\mathcal{T}}(C,A[1])$. For any $\delta\in\mathbb{E}(C,A)$, take a triangle $$A\stackrel{}{\longrightarrow}B\stackrel{}{\longrightarrow}C\stackrel{\delta}{\longrightarrow}A[1]$$
		and define the realization $\mathfrak{s}(\delta)=[A\stackrel{}{\longrightarrow}B\stackrel{}{\longrightarrow}C]$.  Then $(\mathcal{T}, \mathbb{E},\mathfrak{s})$ is an extriangulated category. We refer to \cite[Proposition 3.22]{Na} for more details.

		$(3)$ Let $\mathcal{C}$ be an extension-closed subcategory of $\mathcal{A}$. We define $\mathbb{E}_{\mathcal{C}}$ by the restriction of $\mathbb{E}$ onto $\mathcal{C}^{\rm op}\times \mathcal{C}$ and define $\mathfrak{s}_{\mathcal{C}}$ by restricting $\mathfrak{s}$. One can check directly that $(\mathcal{C},\mathbb{E}_{\mathcal{C}},\mathfrak{s}_{\mathcal{C}})$ is an extriangulated category (cf. \cite[Remark 2.18]{Na}).
	\end{example}
	
	\subsection{Extriangulated length categories} 	We denote by $\Iso(\mathcal{A})$ the set of isomorphism class of objects in $\mathcal{A}$. We often identify an isomorphism class  with its representative.
	
	\begin{definition}\label{D-1} (\cite[Definition 3.1]{Wa}) We say that a map $\Theta:\Iso(\mathcal{A})\rightarrow \mathbb{N}$ is a {\em length function} on $\mathcal{A}$ if it satisfies the following conditions:
		\begin{itemize}
			\item [(1)] $\Theta(X)=0$ if and only if $X\cong0$.
			\item [(2)] For any $\mathbb{E}$-triangle $X\stackrel{}{\longrightarrow}L\stackrel{}{\longrightarrow}M\stackrel{\delta}\dashrightarrow$ in $\mathcal{A}$, we have $\Theta(L)\leq \Theta(X)+\Theta(M)$. In addition, if $\delta=0$, then $\Theta(L)=\Theta(X)+\Theta(M)$.
		\end{itemize}
		For any $M\in\Iso(\mathcal{A})$, the value $\Theta(M)$ is called the {\em length} of $M$.
	\end{definition}
	
	Let $\Theta$ be a length function on $\mathcal{A}$. We say that an $\mathbb{E}$-triangle $$X\stackrel{x}{\longrightarrow}L\stackrel{y}{\longrightarrow}M\stackrel{\delta}\dashrightarrow$$ in $\mathcal{A}$ is {\em $\Theta$-stable} (or {\em stable} when $\Theta$ is clear from the context) if $\Theta(L)=\Theta(X)+\Theta(M)$. In this case,  $x$ is called a {\em $\Theta$-inflation}, $y$ is called a {\em $\Theta$-deflation} and $\delta$ is called a  {\em $\Theta$-extension}. If there is no confusion, we depict $x$ by the symbol $X\rightarrowtail L$ and $y$ by $L\twoheadrightarrow M$.
	
	\begin{remark}\label{R-2-5} Let $X\stackrel{x}\rightarrowtail L\stackrel{y}\twoheadrightarrow M\stackrel{}\dashrightarrow$ be a stable $\mathbb{E}$-triangle. Then $\Theta(X)\leq\Theta(L)$. It is easily checked that $\Theta(X)=\Theta(L)$ if and only if $x$ is an isomorphism. Dually, $\Theta(M)=\Theta(L)$ if and only if $y$ is an isomorphism.
	\end{remark}
	A morphism $f:M\rightarrow N$ in $\mathcal{A}$ is called {\em $\Theta$-admissible} if $f$ admits a {\em $\Theta$-decomposition $(i_{f}, X_{f},j_{f})$}, i.e. there is a commutative diagram
	$$\xymatrix{
		M\ar[rr]^{f}  \ar@{->>}[dr]_{i_{f}}
		&  &   N     \\
		& X_{f}  \ar@{>->}[ur]_{j_{f}}                }
	$$
	such that $i_{f}$ is a $\Theta$-deflation and $j_{f}$ is a $\Theta$-inflation.

	\begin{definition} (\cite[Definition 3.8]{Wa}) Let  $\Theta$  be a  length function on $\mathcal{A}$. We say that $((\mathcal{A}, \mathbb{E},\mathfrak{s}),\Theta)$ is an  {\em  extriangulated length category},  or simply a {\em length category}, if every morphism in $\mathcal{A}$ is $\Theta$-admissible.
		For simplicity, we often write   $(\mathcal{A},\Theta)$  for  $((\mathcal{A}, \mathbb{E},\mathfrak{s}),\Theta)$ when $\mathbb{E}$ and $\mathfrak{s}$ are clear from the context.
	\end{definition}
	
	\begin{remark} For a length category  $(\mathcal{A},\Theta)$, we may
		omit the length function $\Theta$ and simply say that $\mathcal{A}$ is a length category. We will see later in Example \ref{E-2-9} that abelian length categories and bounded derived categories of finite dimensional algebras with finite global dimension are  length categories. 
	\end{remark}

	Let $\mathcal{X}$ be a collection of objects in $\mathcal{A}$. The {\em filtration subcategory} $\mathbf{Filt_{\mathcal{A}}(\mathcal{X})}$ is consisting of all objects $M$ admitting a finite filtration of the form
	\begin{equation*}
		0=M_{0}\stackrel{f_{0}}{\longrightarrow}M_{1}\stackrel{f_{1}}{\longrightarrow}M_{2}{\longrightarrow}\cdots\xrightarrow{f_{n-1}}M_{n}=M
	\end{equation*}
	with $f_{i}$ being an inflation and $\cone(f_{i})\in\mathcal{X}$ for any $0\leq i\leq n-1$.  For each object $M\in\Filt_{\mathcal{A}}(\mathcal{X})$, the minimal length of $\mathcal{X}$-{filtrations} of $M$ is called the $\mathcal{X}$-{\em length} of $M$, which is denoted by $l_{\mathcal{X}}(M)$.  Note that $\Filt_{\mathcal{A}}(\mathcal{X})$ is closed under extensions by \cite[Lemma 2.8]{Wa3}. As stated in Example \ref{E-2-3}(3),  we may regard $\Filt_{\mathcal{A}}(\mathcal{X})$ as an extriangulated category.
	
	\begin{definition}\label{D-2-5}
		An object $M\in\mathcal{A}$ is called a {\em brick} if its endomorphism ring is a division ring. A set $\mathcal{X}$ of isomorphism classes of bricks in $\mathcal{A}$
		is called a {\em semibrick} if $\Hom_{\mathcal{A}}(X_1,X_2)=0$ for any two non-isomorphic objects $X_1,X_2$ in $\mathcal{X}$. If moreover $\mathcal{A}=\Filt_{\mathcal{A}}(\mathcal{X})$, then we say $\mathcal{X}$ is a {\em simple-minded system} in $\mathcal{A}$.
	\end{definition}
	
	For a length category  $(\mathcal{A},\Theta)$, we  define
	$$\Theta_{1}:=\{M\in\Iso(\mathcal{A})~|~0<\Theta(M)\leq \Theta(N)~\text{for any}~0\neq N\in\Iso(\mathcal{A})\}.$$
	For $n\geq2$, we  inductively define various sets  as follows:
	$$\Theta'_{n}=\{M\in\Iso(\mathcal{A})~|~M\in\Theta_{n-1}^{\perp}\bigcap{^{\perp}}\Theta_{n-1},\Theta(M)=n\}~\text{and}~\Theta_{n}=\Theta_{n-1}\bigcup\Theta'_{n}.$$
	Set $\Theta_{\infty}=\bigcup_{n\geq1}\Theta_{n}$. We have the following characterization of length categories.
	
	\begin{theorem}\label{main0}  Let $\mathcal{A}$ be an extriangulated category.
		
		$(1)$ If $\mathcal{X}$ is a simple-minded system in $\mathcal{A}$, then  $(\mathcal{A},l_{\mathcal{X}})$ is a length category.
		
		$(2)$ If $(\mathcal{A},\Theta)$ is a length category, then $\Theta_{\infty}$ is a simple-minded system in $\mathcal{A}$.
		
		$(3)$ $\mathcal{A}$ is a length category if and only if $\mathcal{A}$ has a simple-minded system.
		
	\end{theorem}
	\begin{proof} This follows from \cite[Theorem 3.9 and Proposition 3.13]{Wa}.
	\end{proof}
	
	\begin{example}\label{E-2-9} (1) Let $\mathcal{A}$ be an abelian length category. We denoted by $\cim(\mathcal{A})$ the set of isomorphism classes of simple objects in $\mathcal{A}$. It is straightforward to check that $\cim(\mathcal{A})$ is a simple-minded system in $\mathcal{A}$. Then Theorem \ref{main0}  implies that $\mathcal{A}$ is a length category.
		
		$(2)$ Let $\Lambda$ be a finite dimensional algebra of finite global dimension. It was shown in \cite[Example 2]{Du} that there exists a simple-minded system in bounded derived category $D^{b}(\Lambda)$. Thus $D^{b}(\Lambda)$ is a length category. We refer to \cite[Example 3.25]{Wa} for more details.
	\end{example}
	
	We collect some useful results on length categories.
	\begin{proposition}\label{P-2-6} Let $(\mathcal{A},\Theta)$ be a length category.
		
		$(1)$ For any $A,B\in\mathcal{A}$, we have $\Theta(A\oplus B)=\Theta(A)+\Theta(B)$.
		
		$(2)$ The classes of $\Theta$-inflations $($resp. $\Theta$-deflations$)$ is closed under compositions.
		
		$(3)$ Let $f:X\rightarrow Y$ and  $g:Y\rightarrow Z$ be two morphisms in $\mathcal{A}$. If $gf$ is a $\Theta$-inflation, then so is $f$. Dually, if $gf$ is a $\Theta$-deflation, then so is $g$.

		$(4)$ Take $X\in\Theta_{1}$. If $f:X\rightarrow M$ is a non-zero morphism in $\mathcal{A}$, then $f$ is a $\Theta$-inflation. Dually,  if $g:M\rightarrow X$ is a non-zero morphism in $\mathcal{A}$, then $g$ is a $\Theta$-deflation.

		$(5)$ Suppose that $\Theta_1=\Theta_{\infty}$. For any $M\in\mathcal{A}$, there exists two stable $\mathbb{E}$-triangles $X_{1}\stackrel{}\rightarrowtail M\stackrel{}\twoheadrightarrow M'\stackrel{}\dashrightarrow $ and $M''\stackrel{}\rightarrowtail M\stackrel{}\twoheadrightarrow X_{2}\stackrel{}\dashrightarrow$ such that $\Theta(X_{1})=\Theta(X_{2})=1$.
	\end{proposition}
	\begin{proof}

		(1) It follows immediately from the definition of length function.
		
		(2)--(4) The reader can find the statement (2)	in \cite[Lemma 3.6]{Wa}, (3) in \cite[Lemma 3.20]{Wa} and (4) in \cite[Lemma 3.11]{Wa}.
		
		(5) By \cite[Lemma 3.17]{Wa}, we have $\Theta_1=\Theta_{\infty}$ if and only if $\Theta_1=l_{\Theta_1}$. Then the statement follows from \cite[Lemma 3.5 and Corollary 3.6]{Wa3}.		
	\end{proof}

	\section{Gabriel-Roiter measure}
	First, we recall some notations of poset from \cite{Kr}. Let $(P,\leq)$ be a partially ordered set. A finite {\em chain} of $P$ is a totally ordered subset of $P$. We denote by $\mathrm{Ch}(P)$ the set of finite chains in $P$. For $\mathbf{X}\in\mathrm{Ch}(P)$, we denote by $\min \mathbf{X}$ its minimal element and by $\max \mathbf{X}$ its maximal element. We use the convention that $\max\emptyset< \mathbf{X}<\min\emptyset$ for any $\mathbf{X}\in\mathrm{Ch}(P)$.  On $\mathrm{Ch}(P)$ we consider the {\em lexicographical order} defined by
	$$\mathbf{X}\leq\mathbf{Y}:=~\min(\mathbf{Y}\backslash \mathbf{X})\leq \min(\mathbf{X}\backslash \mathbf{Y})~\text{for}~ \mathbf{X}, \mathbf{Y}\in \mathrm{Ch}(P).$$
	For any $x\in P$, we write
	$$\mathrm{Ch}(P,x)=\{\mathbf{X}\in\mathrm{Ch}(P)~|~\max(\mathbf{X})=x\}.$$
	
	In this section, we always assume that $((\mathcal{A}, \mathbb{E},\mathfrak{s}),\Theta)$  is a length category. Let $X,Y\in\Iso(\mathcal{A})$. We write $X\leq Y$ if there exists a $\Theta$-inflation $X\rightarrowtail Y$. In this case, we say $X$ is a {\em subobject} of $Y$. If moreover   $X\ncong Y$, we say that $X$ is a {\em proper subobject} of $Y$ and write $X<Y$.
	
	\begin{remark}\label{R-3-1}  It is obvious that the $X\leq X$ for any $X\in\Iso(\mathcal{A})$. Let $X,Y,Z\in\Iso(\mathcal{A})$.  If $X\leq Y\leq Z$, then  there exists two $\Theta$-inflations $f:X\rightarrowtail Y$ and  $g:Y\rightarrowtail Z$. By Proposition \ref{P-2-6}(2), the morphism $gf$ is a $\Theta$-inflation and thus $X\leq Z$. If moreover $X=Z$, then $\Theta(X)=\Theta(Y)$. By using Remark \ref{R-2-5}, we conclude that $X=Y$. The observation above implies that $(\Iso(\mathcal{A}),\leq)$ is actually a partially ordered set. 	
	\end{remark}
	We denote by $\ind(\mathcal{A})$ the set of isomorphism classes of indecomposable objects in $\mathcal{A}$. By Remark \ref{R-3-1}, we infer that the set $\ind(\mathcal{A})$ is a partially ordered sets with respect to the  relation $\leq$. Then the length function $\Theta$ induces a map $\Theta:\ind(\mathcal{A})\rightarrow \mathbb{N};~M\mapsto \Theta(M)$ of partial order sets. Following \cite{Kr}, we introduce the Gabriel-Roiter measure for a given length category, which will be useful for resolving  the first Brauer-Thrall conjecture.

	\begin{definition} A  {\em Gabriel-Roiter measure} of $\mathcal{A}$ with respect to $\Theta$ is a function
		$$\Theta^{\ast}:\ind(\mathcal{A})\rightarrow \mathrm{Ch}(\mathbb{N});~X\mapsto \max\{\Theta(\mathbf{X})|~\mathbf{X}\in\mathrm{Ch}(\ind(\mathcal{A}),X)\}.$$
		For any $M\in\ind(\mathcal{A})$, the value $\Theta^{\ast}(M)$ is called the {\em GR measure} of $M$.
	\end{definition}
	
	\begin{remark}
		Let $\mathcal{A}$ be an abelian length category and set $\Theta:=l_{\rm sim(\mathcal{A})}$. Recall from Example \ref{E-2-9}(1) that $(\mathcal{A},\Theta)$ is a length category. The  Gabriel-Roiter measure $\Theta^{\ast}_{H^{\ast}}$ for bounded derived category $D^{b}(\mathcal{A})$ was introduced by Krause  \cite{Kr2} to generalize the  Gabriel-Roiter measure $\Theta^{\ast}$ for $\mathcal{A}$. Explicitly, we have $\Theta=\Theta^{\ast}_{H^{\ast}}{\rm inc}$ for the 
		canonical inclusion ${\rm inc}:\ind(\mathcal{A})\rightarrow\ind(D^{b}(\mathcal{A}))$ (see \cite[Section 5]{Kr2}).
	\end{remark}
	
	\begin{lemma}\label{L-1} Let $X,Y\in\Iso(\mathcal{A})$.
		
		$(1)$ If $X<Y$, then $\Theta(X)<\Theta(Y)$.
		
		$(2)$ Either $\Theta(X)\leq\Theta(Y)$ or $\Theta(Y)\leq \Theta(X)$.
		
		$(3)$ The set $\{\Theta(X')~|~X'\in\Iso(\mathcal{A})~\text{and}~X'\leq X\}$ is finite.
	\end{lemma}
	\begin{proof} This immediately follows from the definition of length function.
	\end{proof}

	\begin{proposition}\label{pro} Let $X,Y\in\ind(\mathcal{A})$. Then the Gabriel-Roiter measure $\Theta^{\ast}:\ind(\mathcal{A})\rightarrow \mathrm{Ch}(\mathbb{N})$ satisfies the following conditions:
		\begin{itemize}
			\item [(GR1)] If $X\leq Y$, then $\Theta^{\ast}(X)\leq \Theta^{\ast}(Y)$.
			\item [(GR2)] If $\Theta^{\ast}(X)=\Theta^{\ast}(Y)$, then $\Theta(X)=\Theta(Y)$.
			\item [(GR3)] If $\Theta(X)\geq \Theta(Y)$ and $\Theta^{\ast}(X')<\Theta^{\ast}(Y)$ for any $X'< X$, then $\Theta^{\ast}(X)\leq \Theta^{\ast}(Y)$.
		\end{itemize}
	\end{proposition}
	\begin{proof} Note that $\Theta$ is a length function in the sense of \cite[Section 1]{Kr} by Lemma \ref{L-1}. Then the proof immediately follows from \cite[Theorem 1.7]{Kr}.
	\end{proof}
	
	We can now state the main property of the  Gabriel-Roiter measure. These result was proved by Gabriel in \cite{Ga} for artin algebras (see also \cite{Ri}), and later by Krause for abelian length categories (cf. \cite[Proposition 3.2]{Kr}).

	\begin{theorem}\label{main1} Let $(\mathcal{A},\Theta)$ be a  length category. Take $X,Y_{1},\cdots, Y_{n}\in\ind(\mathcal{A})$ such that $\Theta^{\ast}(Y_{k})=\max\{\Theta^{\ast}(Y_{i})~|~1\leq i\leq n\}$. Suppose that $f=(f_{1},\cdots,f_{n})^{T}:X\rightarrowtail Y=\bigoplus\limits_{i=1}^{n}Y_{i}$ is a  $\Theta$-inflation. Then $\Theta^{\ast}(X)\leq \Theta^{\ast}(Y_{k})$.  Moreover, if $\Theta^{\ast}(X)=\Theta^{\ast}(Y_{k})$, then there exists an isomorphism $f_s:X\rightarrow Y_s$ such that  $\Theta^{\ast}(Y_s)=\Theta^{\ast}(Y_{k})$.
	\end{theorem}
	\begin{proof}
		We proceed the proof by induction on $\Theta(X)+\Theta(Y)$. If $\Theta(X)+\Theta(Y)=2$, then $X\cong Y$ by Remark \ref{R-2-5}. Thus the assertion follows. Now suppose that $\Theta(X)+\Theta(Y)>2$. We take a $\Theta$-decomposition $(a_{i},Y'_{i},b_{i})$ for each morphism $f_{i}:X\rightarrow Y_{i}$. Recall that $\Theta^{\ast}(Y_{k})=\max\{\Theta^{\ast}(Y_{i})~|~1\leq i\leq n\}$. We consider two cases:
		
		$(\mathbf{Case~1})$ $f_{k}$ is a $\Theta$-deflation. Then $\Theta(X)\geq \Theta(Y_{k})$ since  $\Theta$ is a length function.  Let $X'$ be a proper indecomposable subobject of $X$. Then there exists a $\Theta$-inflation $g:X'\rightarrowtail X$ such that $\Theta(X')<\Theta(X)$. Note that $gf:X'\rightarrow Y$ is a $\Theta$-inflation by Proposition \ref{P-2-6}(2). By induction hypothesis, we have $\Theta^{\ast}(X')\leq \Theta^{\ast}(Y_{k})$. Moreover, if $\Theta^{\ast}(X')= \Theta^{\ast}(Y_{k})$, then there exists an isomorphism $f_sg:X'\rightarrow Y_s$ such that $\Theta^{\ast}(Y_s)=\Theta^{\ast}(Y_{k})$. In this case, $g$ is a section and hence $X'\cong X$. This is a contradiction. Thus we have $\Theta^{\ast}(X')<\Theta^{\ast}(Y_{k})$. Then (GR3) implies that $\Theta^{\ast}(X)\leq \Theta^{\ast}(Y_{k})$. In particular, if  $\Theta^{\ast}(X)= \Theta^{\ast}(Y_{k})$, then $\Theta(X)= \Theta(Y_{k})$ by (GR2). Thus $f_{k}$ is an isomorphism.

		$(\mathbf{Case~2})$ $f_{k}$ is not a $\Theta$-deflation. Since $b_{i}:Y'_{i}\rightarrowtail Y_{i}$ is a $\Theta$-inflation, we have $\Theta(Y'_{i})\leq \Theta(Y_{i})$  for any $1\leq i\leq n$. If $\Theta(Y'_{k})=\Theta(Y_{k})$, then $b_{k}$ is an isomorphism and hence $f_{k}\cong a_{k}$ is a $\Theta$-deflation. This is a contradiction, thus we have $\Theta(Y'_{k})<\Theta(Y_{k})$.  By Proposition \ref{P-2-6}(1), we have
		$$\Theta(\bigoplus\limits_{i=1}^{n} Y_{i}')=\sum\limits_{i=1}^{n}\Theta(Y_{i}')< \sum_{i=1}^{n}\Theta(Y_{i})=\Theta(Y).$$
		Observe that
		$${\rm diag}(b_{1},\cdots,b_{n})(a_{1},\cdots,a_{n})^{T}=(f_{1},\cdots,f_{n})^{T}=f.$$
		Then Proposition \ref{P-2-6}(3) implies that $(a_{1},\cdots,a_{n})^{T}:X\rightarrow \overline{Y}=\bigoplus\limits_{i=1}^{n} Y_{i}'$ is a $\Theta$-inflation. Take a decomposition $Y'_{i}=\bigoplus\limits_{j=1}^{s_{i}}Y_{ij}$ into indecomposable direct summands $Y_{ij}$. Recall that  $\Theta(\overline{Y})<\Theta(Y)$. By induction hypothesis, we have 
		$$\Theta^{\ast}(X)\leq \max\{\Theta^{\ast}(Y_{ij})~|~1\leq i\leq n,1\leq j\leq s_{i}\}\leq \max\{\Theta^{\ast}(Y_{i})~|~1\leq i\leq n\}=\Theta^{\ast}(Y_{k}).$$
		Suppose that $\max\{\Theta^{\ast}(Y_{ij})~|~1\leq i\leq n,1\leq j\leq s_{i}\}=\Theta^{\ast}(Y_{it})$. If $\Theta^{\ast}(X)= \Theta^{\ast}(Y_{k})$, then $\Theta^{\ast}(X)=\Theta^{\ast}(Y_{it})\leq \Theta^{\ast}(Y_{i})\leq \Theta^{\ast}(Y_{k})$ by (GR1). This implies that $\Theta^{\ast}(X)=\Theta^{\ast}(Y_{it})= \Theta^{\ast}(Y_{i})=\Theta^{\ast}(Y_{k})$. By (GR2), we have $\Theta(X)=\Theta(Y_{it})=\Theta(Y_{i})=\Theta(Y_{k})$. Note that $\Theta(Y_{it})\leq \Theta(Y'_{i}) \leq \Theta(Y_{i})$. Then $\Theta(Y'_{i})=\Theta(Y_{i})$ and thus $b_{i}:Y'_{i}\rightarrowtail Y_{i}$ is an isomorphism. It follows that $f_{i}\cong a_{i}$ is a $\Theta$-deflation. Since $\Theta(X)=\Theta(Y_{i})$, we infer  that $f_i:X\rightarrow Y_i$ is an isomorphism.
	\end{proof}
	
	\section{The first Brauer-Thrall conjecture}
	Our aim in this section is to provided a comprehensive answer to the first Brauer-Thrall in the setting of  length categories. In this section, we fix a connected artin algebra $\Lambda$ and denote by $R$ the center of $\Lambda$. We say an extriangulated  category  $(\mathcal{A}, \mathbb{E},\mathfrak{s})$ is {\em R-linear} if $\Hom_{\mathcal{A}}(A,B)$ and $\mathbb{E}(A,B)$ are  $R$-modules for any $A,B\in\mathcal{A}$. We write $\dim_{R}\Hom_{\mathcal{A}}(A,B)$ and $\dim_{R}\mathbb{E}(A,B)$ to denote the length of $\Hom_{\mathcal{A}}(A,B)$ and $\mathbb{E}(A,B)$ as an $R$-module, respectively. We start with the following Ext-lemma, which is a well-known result in  homological algebra (cf. \cite[Section 3]{Ri}).
	
	\begin{lemma}\label{L-2} Let $(\mathcal{A}, \mathbb{E},\mathfrak{s})$ be an $R$-linear extriangulated category. Let $$X^{n}\oplus Y\stackrel{f}{\longrightarrow} L\stackrel{g}{\longrightarrow} M\stackrel{\delta}\dashrightarrow$$
		be an $\mathbb{E}$-triangle in $\mathcal{A}$. Suppose that $X\in\ind(\mathcal{A})$. If $\dim_{R}\mathbb{E}(M,X)< n$, then $fp_{i}$ is a section for some canonical inclusion $p_{i}:X\rightarrow X^{n}\oplus Y$.
	\end{lemma}
	\begin{proof} For each projection map $\pi_{i}:X^{n}\oplus Y\rightarrow X$, there exists a commutative diagram
		\begin{equation*}
			\xymatrix{
				X^{n}\oplus Y\ar[d]_{\pi_{i}} \ar[r]^-{f} &  L\ar[d]_{} \ar[r]^-{g} &  M\ar@{=}[d]_{} \ar@{-->}[r]^{\delta} &   \\
				X\ar[r]^{} & L_{i} \ar[r]^{} &  \ar@{-->}[r]^{{\pi_{i}}_{\ast}(\delta)} M &  . }
		\end{equation*}
		Since $\dim_{R}\mathbb{E}(M,X)< n$, there is a non-trivial linear combination
		$$\sum\limits_{i=1}^{n} \lambda_{i}{\pi_{i}}_{\ast}(\delta)=(\sum\limits_{i=1}^{n} \lambda_{i}{\pi_{i}})_{\ast}(\delta)=0$$
		for some $\lambda_{i}\in R$.
		Thus we have the following commutative  diagram
		\begin{equation*}
			\xymatrix{
				X^{n}\oplus Y\ar[d]_{\sum\limits_{i=1}^{n} \lambda_{i}{\pi_{i}}} \ar[r]^-{f} &  L\ar[d]_{b} \ar[r]^{} &  M\ar@{=}[d]_{} \ar@{-->}[r]^{\delta} &   \\
				X\ar[r]^{a} & L' \ar[r]^{} &  \ar@{-->}[r]^{0} M &  . }
		\end{equation*}
		Since $a$ is a section, there exists a retraction $a':L'\rightarrow X$ such that $a'a=1_{X}$. We may assume that $\lambda_i\neq 0$. Then
		$$a'bfp_{i}=(\sum\limits_{i=1}^{n} \lambda_{i}{\pi_{i}})p_{i}=\lambda_{i}1_{X}$$
		is an isomorphism. This implies that $fp_{i}$ is a section.
	\end{proof}
	
	Now let us return to the length category $(\mathcal{A},\Theta)$. Recall that there exists a Gabriel-Roiter measure $\Theta^{\ast}:\ind(\mathcal{A})\rightarrow \mathrm{Ch}(\mathbb{N})$ with respect to $\Theta$.
	
	\begin{lemma}\label{L-4-1} Take $M\in\ind(\mathcal{A})$.
		
		$(1)$ If $\Theta(M)=1$, then  $\Theta^{\ast}(M)=\{1\}$.
		
		$(2)$ If $\Theta(M)>1$ and  $M\notin\Theta_1^{\perp}$, then $\{1\}<\Theta^{\ast}(M)$.
		
		$(3)$ If $\Theta_1=\Theta_{\infty}$, then $\min\{\Theta^{\ast}(M)~|~M\in\ind(\mathcal{A})\}=\{1\}$. 
	\end{lemma}
	\begin{proof} $(1)$	It is clear that $\mathrm{Ch}(\ind(\mathcal{A}),M)=\{M\}$ and thus $\Theta^{\ast}(M)=\{1\}$. 
		
		$(2)$ By hypothesis,  there exists a non-zero morphism $f:S\rightarrow M$ for some $S\in\Theta_1$.  Then Proposition \ref{P-2-6}(4) implies that $f$ is a $\Theta$-inflation. By using (1) together with (GR1), we conclude that $\{1\}=\Theta^{\ast}(S)<\Theta^{\ast}(M)$. 
		
		$(4)$ By Proposition \ref{P-2-6}(5), we infer that $\Theta_1^{\perp}$ is actually an empty set. Then the assertion follows from (1) and (2). 
	\end{proof}

	\begin{lemma}\label{L-4-3} Let $A\stackrel{}\rightarrowtail B\stackrel{f}\twoheadrightarrow C\stackrel{}\dashrightarrow$ and $A'\stackrel{}\rightarrowtail B'\stackrel{g}\twoheadrightarrow C\stackrel{}\dashrightarrow$ be two $\mathbb{E}$-triangles. Then $(f,g):B\oplus B'\rightarrow C$ is a $\Theta$-deflation.
\end{lemma}
\begin{proof}
By using \cite[Proposition 3.15]{Na} together with \cite[Lemma 3.21]{Wa}, we get the following   commutative diagram of stable $\mathbb{E}$-triangles.
	$$\xymatrix{
	& A' \vphantom{\big|} \ar@{=}[r]^{}  \ar@{>->}[d]^{}  & A'\vphantom{\big|}  \ar@{>->}[d] &\\
	A\vphantom{\big|} ~\ar@{>->}[r] \ar@{=}[d]^{}& P\ar@{->>}[d]^{k} \ar@{->>}[r]^{h} & B' \ar@{->>}[d]^-{g}\ar@{-->}[r]^-{}& \\
	A\vphantom{\big|} 	~\ar@{>->}[r]&  B  \ar@{-->}[d]^-{}\ar@{->>}[r]^{f} & C\ar@{-->}[d]^-{}\ar@{-->}[r]&\\
	&  &  }
$$
By \cite[Lemma 3.6]{Hu}, there exists an $\mathbb{E}$-triangle $P\stackrel{-k\choose h}\rightarrow B\oplus B'\stackrel{(f,g)}\rightarrow C\stackrel{}\dashrightarrow$. Observe that
$$\Theta(B\oplus B')=\Theta(B')+\Theta(B)=\Theta(A')+\Theta(C)+\Theta(P)-\Theta(A')=\Theta(C)+\Theta(P).$$
This finishes the proof.
\end{proof}

We say an $R$-linear length category $((\mathcal{A}, \mathbb{E},\mathfrak{s}),\Theta)$  is {\em $R$-finite} if $\dim_{R}\Hom_{\mathcal{A}}(A,B)<\infty$ and $\dim_{R}\mathbb{E}(A,B)<\infty$ for any $A,B\in\mathcal{A}$.
	\begin{definition} Let $(\mathcal{A},\Theta)$ be an $R$-finite length category with $\Theta_1=\Theta_{\infty}$. We say  $(\mathcal{A},\Theta)$ is of {\em finite type} if $\Theta_{1}$ is a finite set. Otherwise,  we say $(\mathcal{A},\Theta)$ is of {\em infinite type}.
	\end{definition}
	
	\begin{remark}\label{R-2} The most elementary example in representation theory is the finitely generated module category $\mod \Lambda$ for a finite dimensional algebra $\Lambda$ over a field. Denote by $\mathcal{S}$ the set of isomorphism classes of simple $\Lambda$-modules, which is a finite simple-minded system in $\mod \Lambda$. Thus $(\mod \Lambda,l_{\mathcal{S}})$ is a length category of finite type. We refer to Example \ref{ex} for a length category of infinite type.
	\end{remark}

	Let $(\mathcal{A},\Theta)$ be a length category of finite type. For $I\subseteq \mathbb{N}$, we denote by $\mathbb{G}_{I}$ the set of isomorphism classes of indecomposable objects in $\mathcal{A}$ with $\Theta^{\ast}(M)=I$.  Since $\Theta(\ind(\mathcal{A}))$ is totally ordered, the lexicographical order $\leq$ on $\Theta^{\ast}(\ind(\mathcal{A}))$ is totally ordered. By Lemma \ref{L-4-1}(3), the set of all GR measures in $\mathcal{A}$ is a chain (maybe infinite)
	$$I:\{1\}= I_{1}< I_{2}<\cdots<I_{n}<\cdots $$
	such that  $\ind(\mathcal{A})=\bigcup\mathbb{G}_{I_{i}}$. We refer to the chain $I$ as {\em Gabriel-Roiter chain}.  Now we are able to prove the following main result.
	\begin{theorem}\label{main2} Let  $((\mathcal{A}, \mathbb{E},\mathfrak{s}),\Theta)$  be a length category of finite type. Then the following statements are equivalent:
		
		$(1)$ $|\ind(\mathcal{A})|<\infty$.
		
		$(2)$ The set $\{\Theta(M)~|~M\in\ind(\mathcal{A})\}$ has an upper bound.
		
		$(3)$  The Gabriel-Roiter chain ${I_{1}}< {I_{2}}< \cdots$  has an upper bound.
		
		That is, the  first Brauer-Thrall conjecture holds.
	\end{theorem}
	
	\begin{proof} Since $(\mathcal{A},\Theta)$ is of finite type, we may assume that $\Theta_{1}=\{S_{1},\cdots,S_{n}\}$. We divide the proof into the following steps:
		
		$\underline{\mathbf{Step~1}.}$ For any $t\geq2$, we define $$\mathcal{A}_{t}=\{M\in\ind(\mathcal{A})~|~M\notin\bigcup\limits_{1\leq i\leq t-1 }\mathbb{G}_{I_{i}}, M'\in\bigcup\limits_{1\leq i\leq t-1 }\mathbb{G}_{I_{i}}~\text{for any}~M'\in\ind(\mathcal{A})~\text{with}~M'< M \}.$$
		The following proof  is essentially due to Boundedness lemma  (cf. \cite[Section 3]{Ri}). 
		
		Take $M\in\mathcal{A}_{t}$. By Proposition \ref{P-2-6}(5), there exists a stable $\mathbb{E}$-triangle
		$$M'\stackrel{}\rightarrowtail M\stackrel{}\twoheadrightarrow S\stackrel{}\dashrightarrow$$
		such that $\Theta(M')=\Theta(M)-1$. Take a decomposition $M'=\bigoplus\limits_{i=1}^{m}M_{i}^{s_i}$ into indecomposable direct summands $M_{i}$. If $\dim_{R}\mathbb{E}(S,M_{i})< s_{i}$, then $M_{i}$ is a direct summand of $M$ by Lemma \ref{L-2}. This is a contradiction, hence
		$$\Theta(M)=\Theta(M')+1\leq\sum\limits_{i=1}^{m}\dim_{R}\mathbb{E}(S,M_{i})+1.$$
		The observation above implies that the length of objects in $\mathcal{A}_{t}$ is bounded for any $t\geq2$.

		$\underline{\mathbf{Step~2}.}$  Take $M\in\ind(\mathcal{A})$ with $M\notin\bigcup\limits_{1\leq i\leq t-1 }\mathbb{G}_{I_{i}}$. We claim that there exists an indecomposable subobject $M'\leq M$ such that $M'\in \mathcal{A}_{t}$. If $\Theta(M)=2$, then $\Theta(M')=1$ for any $M'< M$. This shows that $M\in\mathcal{A}_{2}$. For $\Theta(M)>2$, it suffices to consider the case of  $M\notin\mathcal{A}_{t}$. In this case, there exists an indecomposable proper subobject $M'< M$ such that $M'\notin\bigcup\limits_{1\leq i\leq t-1 }\mathbb{G}_{I_{i}}$. By induction hypothesis,  there exists an indecomposable subobject $M''\leq M'<M$ such that $M''\in \mathcal{A}_{t}$.

		$\underline{\mathbf{Step~3}.}$   We claim that $I_{t}=\min\{\Theta^{\ast}(M)|~M\in\mathcal{A}_{t}\}$. To see this, we take $N\in\mathcal{A}_{t}$ such that $\Theta^{\ast}(N)=\min\{\Theta^{\ast}(M)|~M\in\mathcal{A}_{t}\}$.  Since $N\in\mathcal{A}_{t}$, we have $I_{t}\leq\Theta^{\ast}(N)$. On the other hand, for any $M\in\mathbb{G}_{I_{t}}$, there exists an indecomposable subobject $M'\leq M$ such that $M'\in \mathcal{A}_{t}$ by Step 2. By (GR1), we have
		$$I_{t}\leq\Theta^{\ast}(N)\leq \Theta^{\ast}(M')\leq \Theta^{\ast}(M)=I_{t}.$$
		This shows that $I_{t}=\Theta^{\ast}(N)$.

		$\underline{\mathbf{Step~4}.}$  We will show that each $\mathbb{G}_{I_{t}}$ is a finite set. For $t=1$, we have $\mathbb{G}_{I_{1}}=\Theta_{1}=\{S_{1},\cdots,S_{n}\}$ by Lemma \ref{L-4-1}(3). Assume that the claim holds for $i\leq t-1$.  By (GR2), the objects in $\mathbb{G}_{I_{t}}$ have the same length. We denote  it by $l$.  Set $\mathcal{N}=\add\bigcup\limits_{1\leq i\leq t-1 }\mathbb{G}_{I_{i}}$.
		
		$\mathbf{Claim~1.}$ We have $\mathbb{G}_{I_{t}}\subseteq\mathcal{A}_{t}$. 	
		
		For any $M\in\mathbb{G}_{I_{t}}$, we have $\Theta^{\ast}(M)=I_{t}$. By Step 2, there exists an indecomposable subobject $M'\leq M$ such that $M'\in \mathcal{A}_{t}$. By Step 3, we get
		$$I_{t}\leq \Theta^{\ast}(M')\leq\Theta^{\ast}(M)=I_{t}.$$
		By using (GR2), we have $\Theta(M')=\Theta(M)$. Then $M\cong M'\in\mathcal{A}_{t}$ and thus $\mathbb{G}_{I_{t}}\subseteq\mathcal{A}_{t}$. 	
			\vspace{3mm}
			
		$\mathbf{Claim~2.}$ For any $M\in\mathbb{G}_{I_{t}}$, there exists  a $\Theta$-deflation $f:M\twoheadrightarrow M'$ such that $f$ is a left $\mathcal{N}$-approximation.
		
		By induction hypothesis, the set $\bigcup\limits_{1\leq i\leq t-1 }\mathbb{G}_{I_{i}}$ is finite. Since $(\mathcal{A},\Theta)$ is $R$-finite, we infer that $\mathcal{N}$ is functorially finite.  For any $M\in\mathbb{G}_{I_{t}}$, there exists a left $\mathcal{N}$-approximation $f:M\rightarrow N$.  We take a $\Theta$-decomposition $(i_f,X_f,j_f)$ of $f$. Note that $j_f:X_f\rightarrowtail N$ is a  $\Theta$-inflation. By using Theorem \ref{main1}, we infer that $X_f\in\mathcal{N}$. Take a morphism $g:M\rightarrow N'$ with $N'\in\mathcal{N}$. Since $f$ is a left $\mathcal{N}$-approximation, there exists a morphism $h:N\rightarrow N'$ such that $g=hf$. Thus $g=hf=hj_fi_f$.  The observation above implies that the $\Theta$-deflation $i_f:M\twoheadrightarrow X_f$ is a left $\mathcal{N}$-approximation. 
		\vspace{2mm}

We define
$$\mathbb{G}_{I_{t},N}=\{M\in\mathbb{G}_{I_{t}}~|~\text{there exists a $\Theta$-deflation $f:M\twoheadrightarrow N$ such that}$$
$$\text{$f$ is a left $\mathcal{N}$-approximation} \}.$$
and  $G=\{N\in\Iso(\mathcal{N})~|~ \Theta(N)<l~\}$.	Note that  $\bigcup\limits_{1\leq i\leq t-1 }\mathbb{G}_{I_{i}}$ is a finite set. Thus $G$ is a finite set. We assume that $G=\{N_1,N_2,\cdots,N_m\}$.
\vspace{2mm}

$\mathbf{Claim~3.}$ We have 
$\mathbb{G}_{I_{t}}=\bigcup\limits_{i=1}^{m}{G}_{I_{t},N_i}.$

Take $M\in\mathbb{G}_{I_{t}}$. By Claim 2, there exists  a $\Theta$-deflation $f:M\twoheadrightarrow M'$ such that $f$ is a left $\mathcal{N}$-approximation. It is obvious that $\Theta(M')<\Theta(M)=l$. 
\vspace{3mm}

By Claim 3, it  suffices to show that each $\mathbb{G}_{I_{t},N_i}$ is  finite.  Take pairwise non-isomorphic objects $M_{1},\cdots,M_{s}$ in $\mathbb{G}_{I_{t},N_i}$. Then $\Theta(M_{1})=\Theta(M_{2})=\cdots=\Theta(M_{s})=l$. For any $1\leq j\leq s$, there exists a $\Theta$-deflation $g_j:M_j\rightarrow N_i$ such that $g_j$ is a left $\mathcal{N}$-approximation. Set $g=(g_1,\cdots,g_s)$. By Lemma \ref{L-4-3},  there exists a stable $\mathbb{E}$-triangle
$$M'\stackrel{f}\rightarrowtail \bigoplus\limits_{j=1}^{s} M_{j}\stackrel{g}\twoheadrightarrow N_{i}\stackrel{}\dashrightarrow$$
in $\mathcal{A}$. Take a decomposition $M'=\bigoplus\limits_{j=1}^{q}H_{j}^{s_{j}}$ into indecomposable direct summands $H_{j}$. 
\vspace{3mm}

$\mathbf{Claim~4.}$ We have $\Theta^{\ast}(H_{j})<I_{t}$ for any $1\leq j\leq q$.

This proof is inspired from Coamalgamation lemma  (cf. \cite[Section 3]{Ri}). By  Theorem \ref{main1}, we have $\Theta^{\ast}(H_{j})\leq I_{t}$ for any $1\leq j\leq q$. Set $f=(f_{1},\cdots,f_{n})^{T}$. Suppose that  $\Theta^{\ast}(H_{1})=I_{t}$. Again by Theorem \ref{main1}, we may assume that $f_1u$ is an isomorphism for the canonical inclusion $u:H_1\rightarrow M'$. For $2\leq j\leq s$, we take a $\Theta$-decomposition $(a_{j},X_j,b_j)$ for $f_ju:H_1\rightarrow M_j$. Note that $\Theta(H_1)=\Theta(M_j)=l$. If $\Theta(X_j)=l$, then $H_1\cong X_j\cong M_j$. This is a contradiction. By using Theorem \ref{main1}, we infer that $X_j\in\mathcal{N}$. Recall that $g_1$ is a left $\mathcal{N}$-approximation. Then there exists a morphism $h_j:N_i\rightarrow M_j$ such that $f_ju=h_jg_1f_1u$. Thus
$$(-g_1f_1)u=(\sum\limits_{j=2}^{s} g_jf_j)u=(\sum\limits_{j=2}^{s} g_jh_j)g_1f_1u$$
and then $-g_1=(\sum\limits_{j=2}^{s} g_jh_j)g_1.$ By Proposition \ref{P-2-6}(3), the morphism $\sum\limits_{j=2}^{s} g_jh_j:N_i\rightarrow N_i$ is a $\Theta$-deflation.  This implies that $\sum\limits_{j=2}^{s} g_jh_j$ is actually an isomorphism. Thus $(g_2,\cdots,g_s):\bigoplus\limits_{j=2}^{s} M_{j}\rightarrow N_i$ is a retraction. Then $N_i$ is a direct summand of $\bigoplus\limits_{j=2}^{s} M_{j}$. This is a contradiction.
\vspace{3mm}

By Claim 4, we infer that $H_{j}\in\bigcup\limits_{1\leq i\leq t-1 }\mathbb{G}_{I_{i}}$ for any $1\leq j\leq q$. Suppose that  $ \dim_R\mathbb{E}(N_i,H_{j})<s_j$ for some $1\leq j\leq q$. Then Lemma \ref{L-2} implies that $H_j\cong M_k$ for some $1\leq k\leq s$.  This is a contradiction. Thus $s_{j}\leq \dim_{R}\mathbb{E}(N_{i},H_{j})$ for any $1\leq j\leq q$.  Note that $\Theta(N_i)<l$. Then
 $$sl=\Theta(M')+\Theta(N_i)=\sum_{j=1}^{q}s_{j}\Theta(H_{j})+\Theta(N_i)< \sum_{j=1}^{q}\dim_{R}\mathbb{E}(N_i,H_{j})\Theta(H_{j})+l.$$
Recall that $\mathbb{G}_{I_{t}}=\bigcup\limits_{i=1}^{m}{G}_{I_{t},N_i}$ and  $\bigcup\limits_{1\leq i\leq t-1 }\mathbb{G}_{I_{i}}$ is  finite. Set $|\bigcup\limits_{1\leq i\leq t-1 }\mathbb{G}_{I_{i}}|=h$ and 
$$e=\max\{\dim_{R}\mathbb{E}(N_{i},K)~|~1\leq i\leq m~\text{and}~K\in\bigcup\limits_{1\leq i\leq t-1 }\mathbb{G}_{I_{i}}\}.$$ 
 By Claim 1, we have $\mathbb{G}_{I_{i}}\subseteq\mathcal{A}_{i}$ for any $i\geq 1$. By Step 1, the length of objects in $\bigcup\limits_{1\leq i\leq t-1 }\mathbb{G}_{I_{i}}$ has an upper bound $l'$.  Since each $H_{j}\in\bigcup\limits_{1\leq i\leq t-1 }\mathbb{G}_{I_{i}}$, we have $\Theta(H_j)\leq l'$ and $q\leq h$. We conclude that
$$s<\frac{ \sum\limits_{j=1}^{q}\dim_{R}\mathbb{E}(N_i,H_{j})\Theta(H_{j})}{l}+1\leq \frac{ hel'}{l}+1.$$
		This implies that $|\mathbb{G}_{I_{t},N_{i}}|<\frac{ hel'}{l}+1$ and thus $|\mathbb{G}_{I_{t}}|< m\frac{ hel'}{l}+m$.

		$\underline{\mathbf{Step~5}.}$   Now, we  are ready to prove the first Brauer-Thrall conjecture.
		
		$(1)\Rightarrow(2)$: Obvious.
		
		$(2)\Rightarrow(3)$: For $t\geq1$, we define $\mathbf{L}_{t}=\{M\in\ind(\mathcal{A})~|~\Theta(M)=t\}$. On the one hand, we have   $\Theta^{\ast}(M)\subseteq\{1,2,\cdots,t\}$ for any $M\in\mathbf{L}_{t}$. Thus there are only finitely many possible GR measures for $\mathbf{L}_{t}$. On the other hand, the objects in each $G_{I_{i}}$ have the same length. It follows that the Gabriel-Roiter chain has an upper bound.
		
		$(3)\Rightarrow(1)$: By Step 4, we have $|\ind(\mathcal{A})|=|\bigcup\limits_{i=1}^{m}{G}_{I_{t},N_i}|<\infty$.
	\end{proof}

	Recall that  length categories correspond precisely to those categories arising from simple-minded systems. By this, we can  give another version of the Theorem \ref{main2}.
	\begin{corollary} Let $(\mathcal{A}, \mathbb{E},\mathfrak{s})$ be an $R$-finite extriangulated category. For a finite semibrick $\mathcal{X}$, the following conditions are equivalent:
		
		$(1)$ $|\ind(\Filt_{\mathcal{A}}(\mathcal{X})|<\infty$.
		
		$(2)$ The set $\{l_{\mathcal{X}}(M)~|~M\in\ind(\Filt_{\mathcal{A}}(\mathcal{X}))\}$ has an upper bound.
	\end{corollary}

	The first Brauer-Thrall conjecture has been proved by Roiter in \cite{Ro} for finite-dimensional algebras and refined by  Ringel in \cite{Ri}. As a special case of Theorem \ref{main2}, we can recover this well-known fact as follows.
	
	\begin{corollary} \text{\rm (\cite{Ro},\cite{Ri})}  Let $\Lambda$ be a finite dimension algebra over a field. Then $\Lambda$ is of bounded representation type if and only if $\Lambda$ is of finite representation type.
	\end{corollary}
	\begin{proof} This immediately follows from Remark \ref{R-2} and Theorem \ref{main2}.
	\end{proof}

	For infinite type, we provide a counter-example for the first Brauer-Thrall conjecture.
	\begin{example}\label{ex}
		Let $\Lambda$ be the path algebra of the quiver $1\longrightarrow2\longrightarrow3$. The Auslander-Reiten quiver $\Gamma$ of the bounded derived category $D^{b}(\Lambda)$ is as follows:
		
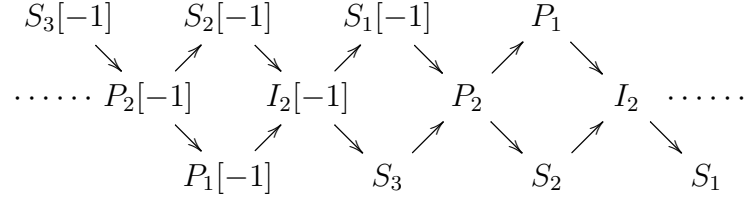
\begin{equation*}
			\xymatrix@!=0.5pc{
				&& S_3[-1]\ar[dr]  && S_2[-1]\ar[dr] && S_1[-1]\ar[dr] && P_1\ar[dr] && \\
				&&\cdots\cdots\quad& P_2[-1]\ar[dr]\ar[ur] && I_2[-1]\ar[ur]\ar[dr] && P_2\ar[ur]\ar[dr] && I_2\ar[dr] & \cdots\cdots\\
				&&&& P_1[-1]\ar[ur] && S_3\ar[ur] && S_2\ar[ur] && S_1}
		\end{equation*}
		Let $\mathcal{X}$ be the set consisting of the isomorphism classes of objects in the top row of $\Gamma$, i.e. $$\mathcal{X}=\bigcup_{i=2k,k\in \mathbb{Z}}\{P_{1}[i-1],S_{3}[i],S_{2}[i],S_{1}[i]\}.$$
		Clearly, $(D^{b}(\Lambda),l_{\mathcal{\mathcal{X}}})$ is a length category of infinite type and $|\ind(D^{b}(\Lambda))|=\infty$. However, we have $l_{\mathcal{\mathcal{X}}}(M)\leq 3$ for any $M\in\ind(D^{b}(\Lambda))$. Thus the first Brauer-Thrall conjecture fails in $D^{b}(\Lambda)$.
	\end{example}

	We finish this section with a straightforward example illustrating Theorem \ref{main2}.
	\begin{example}Keep the notation used in Example  \ref{ex} and set $\mathcal{Y}=\{P_{1}[-1],S_{3},S_{2}\}.$ Then  the Auslander-Reiten quiver of $\mathcal{A}:=\Filt_{D^{b}(\Lambda)}(\mathcal{Y})$ is given by
		
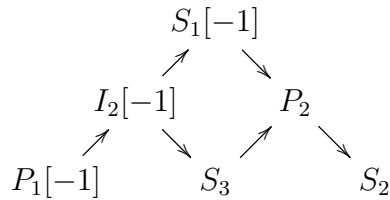
\begin{equation*}
			\xymatrix@!=0.5pc{
				& &  S_1[-1]\ar[dr]^{} & &\\
				&I_2[-1] \ar[dr]\ar[ur]  & &  P_2\ar[dr]^{}&     \\
				P_{1}[-1]\ar[ur]  &  &  S_3\ar[ur]^{} & &  S_2  }
		\end{equation*}
		By this, we obtain a length category  $(\mathcal{A},l_{\mathcal{Y}})$ of finite type. Let us list all 6 indecomposable objects, the corresponding lengths and GR measures as follows:
		$$
		\begin{tabular}{|p{4.1cm}|p{1.2cm}|p{2.2cm}|}
			\hline
			indecomposable object& length&  GR measure\\
			\hline
			$P_{1}[-1]$&  1  &   \{1\}\\
			\hline
			$S_{3}$&    1          & \{1\}  \\
			\hline
			$S_{2}$&      1                &  \{1\}   \\
			\hline
			$I_{2}[-1]$&     2       &   \{1,2\}    \\
			\hline
			$P_{2}$&      2       &   \{1,2\}  \\
			\hline
			$S_{1}[-1]$&      3           &   \{1,2,3\}    \\
			\hline
		\end{tabular}
		$$
		The Gabriel-Roiter chain of the form $\{1\} <\{1,2\} <\{1,2,3\}.$
	\end{example}

	\section*{Acknowledgments}
	Li Wang is supported by the National Natural Science Foundation of China (Grant No. 12301042) and the  Natural Science Foundation of Universities of Anhui (No. 2023AH050904). Jiaqun Wei is supported by the National Natural Science Foundation of China (Grant No. 12271249) and the Natural Science Foundation of Zhejiang Province (Grant No. LZ25A010002).

\end{document}